\documentclass[12pt]{amsart}
\usepackage{graphicx}
\usepackage{amsfonts}
\usepackage{amsmath}
\usepackage{amsthm}
\usepackage{mathrsfs}
\usepackage[v2]{xy}
\usepackage{amscd}  
\usepackage{amssymb}
\usepackage{graphicx}
\usepackage{enumerate}
\usepackage{color}
\usepackage{setspace}
\usepackage{tikz-cd}
\usepackage{hyperref}
\usepackage[margin=1in]{geometry}
\usepackage{setspace}

\usepackage[style = numeric]{biblatex}
\newtheorem{theorem}{Theorem}[section]

\newtheorem{lemma}[theorem]{Lemma}
\newtheorem{proposition}[theorem]{Proposition}
\newtheorem{conjecture}[theorem]{Conjecture}

\theoremstyle{definition}

\theoremstyle{remark}
\newtheorem{remark}[theorem]{Remark}

\newtheorem*{acknowledgements}{Acknowledgements}

\numberwithin{equation}{section}

\def\C{{\mathbb C}}

\def\Q{{\mathbb Q}} 
\def\Z{{\mathbb Z}}

\def\F{{\mathbb F}}
\def\D{{\mathcal D}}
\def\I{{\mathcal I}}

\renewcommand{\P}{\mathbb{P}}

\newcommand{\Mon}{\mathrm{Mon}}
\newcommand{\Aut}{\mathrm{Aut}}

\newcommand{\Gal}{\mathrm{Gal}}
\newcommand{\IMG}{\mathrm{IMG}}

\newcommand{\mf}{\mathfrak}
\renewcommand{\O}{\mathcal{O}}
\newcommand{\A}{\mathbb{A}}

\newcommand{\kbar}{\overline{k}}

\title{A Dynamical Lifting Problem For Additive Polynomials}
\author[D.~Tedeschi]{Daniel Tedeschi}

\begin{document}

\doublespacing

\begin{abstract}
We introduce a dynamical analogue of the lifting problem for Galois covers of algebraic curves and find a negative solution for the collection of separable additive and semi-additive polynomials over $\overline{\mathbb{F}}_p$. We also explicitly compute the dimension of the space of linear conjugacy classes in $\mathcal{M}_{p^m}(\overline{\mathbb{F}}_p)$ which contain additive or semi-additive polynomials providing explicit counterexamples to Thurston Rigidity in positive characteristic.
\end{abstract}


\maketitle

\section{Introduction}

Let $k$ be a perfect field and $f(z) \in k(z)$ a rational map of degree $d \geq 2$. We consider $f$ as a dynamical system, studying the behavior its iterates $f^n = f \circ \cdots \circ f$, with the convention $f^0(z) = z$. The sequence of iterates $\{f^n\}$ forms an infinite tower of branched self-covers of the projective line $f^n: \P^1_k \to \P^1_k$. Two rational maps $f,g$ are \textit{linearly conjugate} if there exists a change of coordinates $\varphi \in \Aut(\P^1_k)$ such that $f^\varphi = \varphi^{-1} \circ f \circ \varphi = g$. A key part of the study of dynamics of rational maps over any field is the characterization of combinatorial and algebraic objects associated to rational maps which are invariant under linear conjugacy. In the present paper, we consider $k = \overline{\F}_p$ and find concrete examples of dynamical invariants which arise over $\overline{\F}_p$ but not $\C$. We also establish the failure of well-known finiteness properties for rational maps over $\C$ in two families of polynomials over $\overline{\F}_p$.

Fix an algebraic closure $\overline{k}$ of $k$. We say a rational map $f(z) \in k(z)$ is \textit{wildly ramified} if all of its iterates are wildly ramified as covers of $\P^1_{\kbar}$, i.e. if the ramification indices of the critical points of $f$ are divisible by the characteristic $p$ of $k$. Otherwise, we say $f$ is \textit{tamely ramified}. Heuristically, the dynamical behavior of a tamely ramified rational map in positive characteristic should closely mirror that of a rational map in characteristic zero. Thus, we turn to wildly ramified maps for examples of behaviors unique to positive characteristic. One way to detect said `uniqueness' is to consider which maps in positive characteristic are the reduction of a map in characteristic zero with the `same' dynamical invariants, i.e. which rational maps admit \textit{lifts}. The most thorough discussion of liftability for dynamical systems comes from Pink \cite{Pink13-Lifting}, in which he establishes liftability for post-critically finite quadratic rational maps which are tamely ramified. The question of existence of lifts of dynamical systems with their dynamical invariants is closely related to the lifting question for Galois covers of algebraic curves, see \cite[\S 9]{HarbaterEtAl18}.

A polynomial $f(z) \in \overline{\F}_p[z]$ is \textit{additive} if $f(a + b) = f(a) + f(b)$ for all $a,b \in \overline{\F}_p$. Additive polynomials necessarily have degree $p^m$ and have exactly one ramification point of index $p^m$, making them wildly ramified. We compute the geometric iterated monodromy group of these polynomials and find they do not admit lifts to characteristic zero which preserve this dynamical invariant.

\begin{theorem}\label{thm:additiveIMG}
Let $f(z) \in \overline{\F}_p[z]$ be an additive, separable polynomial of degree $p^m$. For each $n\geq 1$, we have an isomorphism
\[\Mon_{\overline{\F}_p}(f^n) \cong (\Z/p\Z)^{mn}.\]
The geometric iterated monodromy group then has the form
\[\IMG_{\overline{\F}_p}(f) \cong \varprojlim_{n\geq 1}(\Z/p\Z)^{mn}.\]
Moreover, $f$ does not admit a lift over any characteristic zero DVR which preserves its geometric iterated monodromy group.
\end{theorem}

Similar arguments can be used for \textit{semi-additive} polynomials over $\overline{\F}_p[z]$, which take the form $g(z) = z^n \circ f(z)$ for a suitable additive, separable polynomial $f(z) \in \overline{\F}_p[z]$.

\begin{theorem}\label{thm:introsemiadditive}
Let $n \geq 1$ such that $p \nmid n$. Let $e$ denote the order of $p$ in $(\Z/n\Z)^\times$. Let $f(z) \in \overline{\F}_p[z]$ be an additive polynomial of the form
\[f(z) = \sum_{i=0}^\ell a_i z^{p^{ie}},\]
where $\ell \geq 1$, each $a_i \in \overline{\F}_p$ and $a_\ell, a_0 \neq 0$. Let $m = \ell e$ and suppose $(m,n) \notin \{(1,2), (2,3)\}$. Then $g(z) = z^n \circ f(z)$ does not admit a lift over any characteristic zero DVR which preserves its geometric iterated monodromy group.
\end{theorem}

The computation of iterated monodromy groups is a key problem in arithmetic dynamics, and goes back to foundational works of Odoni \cite{Odoni-2,Odoni-3, Odoni-1}. See \cite{AdamsHyde25,Benedetto17,Benedetto23} for a small sample of these calculations, as well as Jones \cite{Jones14} for a survey of the area. Failure to lift for both additive and semi-additive polynomials relies on an understanding of the finite subgroups of $\Aut(\P^1_k)$. Additive and semi-additive polynomials have monodromy groups which arise as subgroups of $\Aut(\P^1_k)$ in positive characteristic, but not in characterisitic zero. The classification of finite subgroups of $\Aut(\P^1_k)$ with order coprime to the characteristic of $k$ is a classical question, see Beauville \cite{Beauville10}. For a classification of finite subgroups of $\Aut(\P^1_k)$ with order divisible by the characteristic of $k$, see Faber \cite{Faber23}.

We propose the following conjecture which accounts for the failure of liftability for both additive and semi-additive polynomials. This is a natural dynamical analog to the Oort Conjecture for Galois covers of algebraic curves proven by Pop, Obus, and Wewers \cite{Obus-Wewers14, Pop14}.

\begin{conjecture}[Dynamical Oort Conjecture]\label{conj:dynamicaloort}
Let $k$ be a perfect field of positive characteristic. Let $f(z) \in k(z)$ be a PCF rational map of degree $d \geq 2$ satisfying $f'(z) \neq 0$. Let $\mf{p}_1, \dots, \mf{p}_r$ denote the prime ideals of $\kbar(t)$ corresponding to the post-critical points of $f$. For each $1\leq i \leq r$, fix a prime ideal $\tilde{\mf{q}}_i$ of $\kbar \cdot K_\infty$ lying over $\mf{p}_i$. If all of the inertia groups $\mathcal{I}(\tilde{\mf{q}}_i/\mf{p}_i)$ are pro-cyclic then $f$ admits a lift to characteristic zero which preserves its geometric iterated monodromy group.
\end{conjecture}

Along with the failure of liftability comes the failure of finiteness, a core property of PCF maps in characteristic zero. Over $\C$, Thurston Rigidity states that away from a well-understood $1$-dimensional family, a linear conjugacy class of rational maps is determined up to finitely many choices by the graph structure of its critical points and post-critical orbit \cite{DouadyHubbard93}. Every additive polynomial of degree $p^m$ over $\overline{\F}_p$ has the same graph structure for its post-critical orbit. Despite this, we show that the dimension of the space of all linear conjugacy classes containing an additive polynomial grows without bound. Wild ramification in this family yields a striking break from the behavior of dynamical systems in characteristic zero.

\begin{theorem}
Fix $m\geq 1$. Let $\mathcal{A}_m(\overline{\F}_p)$ denote the collection of linear conjugacy classes of degree $p^m$ rational maps over $\overline{\F}_p$ which contain an additive, separable polynomial. We let $\mathcal{A}_m \subset \mathcal{M}_{p^m}$ denote the corresponding subscheme of the moduli space of degree $p^m$ dynamical systems, as described at the end of Section \ref{section:additivepolys}. Then $\dim(\mathcal{A}_m) = m$.
\end{theorem}

Separable, semi-additive polynomials also have a post-critical orbit whose graph structure depends only on their degree. In a similar break from characteristic zero behavior, we find that the dimension of this family of maps grows without bound.

\begin{theorem}
Fix $n, \ell \geq 1$ with $p\nmid n$. Let $e$ denote the order of $p$ in $(\Z/n\Z)^\times$. Let $\mathcal{A}_{n, \ell}(\overline{\F}_p)$ denote the collection of linear conjugacy classes of degree $p^{\ell e}n$ rational maps over $\overline{\F}_p$ which contain a separable, $e$-divisible, semi-additive polynomial (i.e. those described in Theorem \ref{thm:introsemiadditive}). We let $\mathcal{A}_{n, \ell} \subset \mathcal{M}_{p^{\ell e}n}$ denote the corresponding subscheme of the moduli space of degree $p^{\ell e}n$ dynamical systems, as described at the end of Section \ref{section:additivepolys}. Then $\dim(\mathcal{A}_{n, \ell}) = \ell$.
\end{theorem}

In Section \ref{section:background}, we provide the necessary background on dynamical invariants and the moduli space of dynamical systems. Section \ref{section:profiniteinertia} contains the terminology necessary to discuss the sequence of iterates $\{f^n\}$ as a sequence of branched covers of the projective line. In Section \ref{section:dynamicallifts}, we relate the lifting problem for dynamical systems to the lifting problem for Galois covers of algebraic curves and develop Conjecture \ref{conj:dynamicaloort}. In Section \ref{section:additivepolys}, we prove the main results of the paper, that liftability and finiteness fail in the families of additive and semi-additive polynomials. In Section \ref{section:GMlifts}, we investigate the collection of degree $p$ additive polynomials further and compare their dynamical behavior with explicit lifts constructed by Green and Matignon.

\begin{acknowledgements}
The author would like to thank Rachel Pries for asking the question which incited this investigation, and for the many helpful conversations that followed. The author would also like to thank the referee for their careful reading and insightful suggestions.
\end{acknowledgements}

\section{Dynamical Background}\label{section:background}

Due to Silverman \cite{Silverman98}, there exists a moduli space $\mathrm{Rat}_d$ of degree $d$ rational maps and a moduli space $\mathcal{M}_d$ of dynamical systems of degree $d$. Let $k$ be a perfect field with algebraic closure $\kbar$. Then elements of $\mathcal{M}_d(\kbar)$ correspond exactly to equivalence classes of degree $d$ rational maps in $\kbar(z)$ up to linear conjugacy by elements of $\Aut(\P^1_{\kbar})$. We let $[\cdot] : \mathrm{Rat}_d \to \mathcal{M}_d$ denote the morphism of schemes taking a rational map to its linear conjugacy class.

The \textit{post-critical orbit} of $f$, denoted $P_f$, is the union of the forward orbits of the critical points of $f$ in $\P^1(\overline{k})$. Combining $P_f$ with the set of critical points of $f$ yields a weighted-directed graph, whose weights satisfy some simple constraints determined by the geometry of $f$. This structure is often referred to as a \textit{mapping scheme}. See Brezin, et al. \cite{Brezin00} for a discussion of mapping schemes over $\C$, the details of which quite easily extend to any algebraically closed field. Standard algebraic number theory results imply that $P_f$ is precisely the branch locus of the collection of iterates $\{f^n\}$. We say that $f$ is \textit{post-critically finite (PCF)} whenever $P_f$ is finite.

Let $t$ be a transcendental over $k$. For each $n\geq 1$, let $K_n = k(f^{-n}(t))$ denote the splitting field of $f^n(z) - t$. Suppose $f'(z) \neq 0$, so that each $K_n$ is Galois over $k(t)$. Let $\Mon_k(f^n) = \Gal(K_n/k(t))$ denote the \textit{monodromy group} of $f^n$ over $k$. The extensions $K_n$ form an infinite tower with union $K_\infty = \bigcup_n K_n$. The extension $K_\infty$ is Galois over $k(t)$, and we refer to its Galois group as the \textit{(profinite) iterated monodromy group} of $f$ over $k$, denoted $\IMG_k(f)$\footnote{This construction is distinct from the discrete iterated monodromy group found in topology; see Nekrashevych \cite{Nekrashevych03}. Over $\C$ both objects exist and the algebraic construction is the profinite closure of the topological one.}. The iterated monodromy group is naturally isomorphic to the inverse limit $\varprojlim \Mon_k(f^n)$. We refer to $\IMG_{\overline{k}}(f)$ as the \textit{geometric} iterated monodromy group of $f$.

\section{Profinite Inertia and Iterated Monodromy}\label{section:profiniteinertia}

Let $k$ be a perfect field. Let $f(z) \in k(z)$ have degree $d \geq 2$ and suppose $f'(z) \neq 0$. Fix an element $t$ transcendental over $k$. Fix a separable closure $k(t)^{sep}$ of $k(t)$. For each $n \geq 1$, we choose $\theta_n \in k(t)^{sep}$ such that
\[k(\theta_n) \cong \frac{k(t)[z]}{(f^n(z) - t)}.\]
In fact, we can choose the sequence $\{\theta_n\}$ to be compatible, in the sense that $f(\theta_i) = \theta_{i-1}$ for each $i \geq 2$. Let $K_n$ denote the Galois closure of $k(\theta_n)$, i.e. the splitting field of $f^n(z) - t$. 

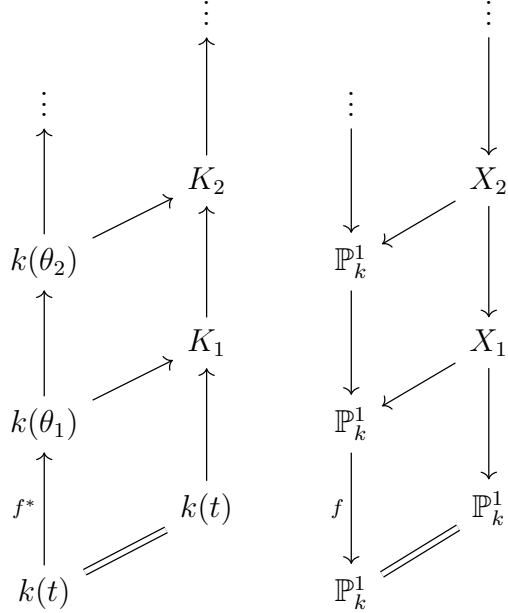
\begin{figure}[h]
\[\begin{tikzcd}[row sep = tiny]
 & \vdots   &  & \vdots \ar[dd]\\
\vdots &   &   \vdots \ar[dd] & \\
 & K_2 \ar[uu]   &   & X_2 \ar[dl] \ar[dd]\\
k(\theta_2) \ar[uu] \ar[ur] &    &  \P^1_k \ar[dd] & \\
 & K_1 \ar[uu]   &    & X_1 \ar[dd] \ar[dl]\\
k(\theta_1) \ar[uu] \ar[ur] &   &   \P^1_k \ar[dd, "f"'] &  \\
 & k(t) \ar[uu]   &   & \P^1_k \\
k(t) \ar[uu, "f^*"] \ar[ur, equal] &    &   \P^1_k \ar[ur, equal] &
\end{tikzcd}\]
\caption{Algebraic and Geometric Perspectives}
\label{fig:alggeoperspective}
\end{figure}

Figure \ref{fig:alggeoperspective} demonstrates the towers being considered from the perspective of function fields and algebraic curves. We let $X_i$ denote the integral proper normal curve over $k$ with function field $K_i$. The choice of primitive element $\theta_n \in k(t)^{sep}$ corresponds to a choice of coordinate on the domain of $f^n: \P^1_k \to \P^1_k$. Given a prime $\mf{p} \subset k(t)$ with corresponding point $P \in \P^1_k$, the discrete data of the factorization of $\mf{p}$ in $k(\theta_n)$ corresponds exactly to the discrete data of the fiber $f^{-n}(P)$. The monodromy group of $f^n$, denoted $\Mon_k(f^n)$, is canonically isomorphic to the Galois group of the cover $\Gal(X_n/\P^1_k)$. These are equally useful perspectives, and an anti-equivalence of categories allows us to move between them freely \cite[Proposition 4.4.5]{Szamuely09}.

Fix a prime ideal $\mf{p}\subset k(t)$. The Galois group $\Gal(K_n/k(t)) = \Mon_k(f^n)$ acts on the prime ideals of $K_n$ lying over $\mf{p}$. Fix such a prime $\mf{q}$ of $K_n$. The \textit{decomposition subgroup} $\D(\mf{q}/\mf{p})$ is the stabilizer of $\mf{q}$ under the action of $\Mon_k(f^n)$. Reduction modulo $\mf{p}$ induces a natural surjection of $\D(\mf{q}/\mf{p})$ onto the Galois group of the extension of residue fields $\Gal(\kappa(\mf{q})/\kappa(\mf{p}))$. The kernel of this surjection is the \textit{inertia subgroup} $\I(\mf{q}/\mf{p})$. When $\mf{p}$ is unramified in $K_n$ then each of the inertia groups at the primes lying over $\mf{p}$ are trivial \cite[Proposition 9.6]{Neukirch99}.  

Let $K_\infty = \bigcup_{n\geq 1}K_n$ be the smallest extension of $k(t)$ containing all of the sets $f^{-n}(t)$. As in the finite case, the Galois group $\IMG_k(f)$ acts on the prime ideals of $K_\infty$ lying over $\mf{p}$. Fix such a prime ideal $\tilde{\mf{q}}$. Let $\D(\tilde{\mf{q}}/\mf{p})$ denotes the stabilizer of $\tilde{\mf{q}}$ under this action, which we may call the (profinite) decomposition subgroup of $\tilde{\mf{q}}$ over $\mf{p}$. Elements of $\D(\tilde{\mf{q}}/\mf{p})$ naturally surject onto automorphisms of the extension of residue fields $\kappa(\tilde{\mf{q}})/\kappa(\mf{p})$. Let $\I(\tilde{\mf{q}}/\mf{p})$ denote the kernel of this surjection, what we may call the (profinite) inertia subgroup of $\tilde{\mf{q}}$ over $\mf{p}$. Suppose $\tilde{\mf{q}}$ corresponds to a sequence $\{\mf{q}_i\}$ of prime ideals $\mf{q}_i \subset K_i$. Recall the following isomorphism: 
\[\Gal(K_\infty/k(t)) \cong \varprojlim \Gal(K_n/k(t)).\]
One can use this fact to show that
$\D(\tilde{\mf{q}}/\mf{p}) \cong \varprojlim \D(\mf{q}_i/\mf{p})$
and $\I(\tilde{\mf{q}}/\mf{p}) \cong \varprojlim \I(\mf{q}_i/\mf{p})$.
In particular, if $\mf{p}$ is unramified in each extension $K_n$, then the inertia groups at the primes of $K_\infty$ lying over $\mf{p}$ are all trivial. In this way, the collection of iterates $\{f^n\}$ gives rise to a profinite Galois branched cover of $\P^1_k$ which is ramified over the post-critical orbit $P_f$.

\section{A Dynamical Lifting Problem}\label{section:dynamicallifts}

There is a well-established lifting problem that arises when studying Galois covers of algebraic curves. See Harbater et al. \cite[\S 9]{HarbaterEtAl18} for a broad overview of the literature and Obus \cite{Obus12} for an in-depth discussion of the results. We include the details necessary to motivate a dynamical analogue, restricting to the case of genus $0$ covers of $\P^1$ for the sake of brevity.

Let $k$ be a field of characteristic $p$. Let $R$ be a characteristic zero DVR with residue field $k$ and fraction field $K$. Let $f(z) \in k(z)$ be a rational map of degree $d \geq 2$ with $f'(z) \neq 0$, considered as a branched cover of $\P^1_k$. Fix a coordinate $t$ on $\P^1_k$ and $T$ on $\P^1_K$. Let $G = \Mon_k(f)$. We say $f$ admits a \textit{$G$-lift (over $R$)} if there exists $\tilde{f}(z) \in R(z)$ that satisfies the following:
\begin{enumerate}
\item the special fiber $\tilde{f} \times_R k \cong f$ as covers of $\P^1_k$;
\item the specialization map $\tilde{f}^{-1}(T) \to f^{-1}(t)$ induces an equivalence between the action of $\Mon_K(\tilde{f})$ on $\tilde{f}^{-1}(T)$ and the action of $\Mon_k(f)$ on $f^{-1}(t)$.
\end{enumerate}
Recall the actions of a group $G$ on a set $\Omega$ and $G'$ on $\Omega'$ are \textit{equivalent} if there exists a group isomorphism $\phi: G \to G'$ and a bijection $f: \Omega \to \Omega'$ such that
for all $g\in G$ and $\omega \in \Omega$ we have the following:
\[f(g*\omega) = \phi(g)*f(\omega).\]

Work in this area makes liberal use of a local-global principle known as formal patching, which is due to Harbater \cite{Harbater03}. It is known that $f$ admits a $G$-lift if and only if the restriction of the cover to each branch point $P$ admits an $I_P$-lift, where $I_P$ denotes the inertia group at any point of the fiber over $P$. A clear dynamical analogue of this result is unknown to the author, though one might expect the profinite nature of the objects involved to make the typical patching techniques unwieldy.

It is a noteworthy result of Grothendieck that tamely ramified Galois covers admit a $G$-lift over the Witt vector ring $W(k)$ \cite[Proposition 2.2]{Obus12}. A well-known conjecture of Oort, and now theorem of Obus, Wewers, and Pop, states that Galois covers admit a $G$-lift whenever all of their inertia groups are cyclic. This is a property enjoyed by all tamely ramified covers, as well as some well-behaved wildly ramified covers. These results can be extended to tame non-Galois covers, as the Galois closure of a tame cover is always tame \cite[Corollary 7.9]{Neukirch99}.

\begin{theorem}[Oort Conjecture]\cite{Obus-Wewers14, Pop14}
Let $\mf{p}_1, \dots, \mf{p}_r$ be the prime ideals corresponding to the branch points of $f$. Fix prime ideals $\mf{q}_i/\mf{p}_i$ in $K_1$, the splitting field of $f(z) - t$ over $k(t)$. Then $f$ admits a $G$-lift to characteristic zero whenever the inertia groups $\mathcal{I}(\mf{q}_i/\mf{p}_i)$ are cyclic for each $1 \leq i \leq r$.
\end{theorem}

One can consider dynamical systems as infinite towers of covers of $\P^1_k$ given by iteration. In doing so, we arrive at a natural dynamical analogue of the above lifting problem for covers.

Let $f(z) \in k(z)$ be a polynomial of degree $d \geq 2$ such that $f'(z) \neq 0$. We say $f$ admits a \textit{geometric $\IMG$-lift (over $R$)} if there exists $\tilde{f}(z) \in R(z)$ that satisfies the following:
\begin{enumerate}
\item the special fiber $\tilde{f} \times_R k \cong f$ as dynamical systems;
\item the specialization map $\O_{\tilde{f}}^-(T) \to \O_f^-(t)$ induces an equivalence between the action of $\IMG_{\overline{K}}(\tilde{f})$ on $\O_{\tilde{f}}^-(T)$ and the action of $\IMG_{\overline{k}}(f)$ on $\O_f^-(t)$.
\end{enumerate}
Note that $\tilde{f} \times_R k \cong f$ implies $\tilde{f}^n \times_R k \cong f$ for all $n$, as reduction modulo the maximal ideal of $R$ commutes with the application of the rational map $\tilde{f}$.

\begin{remark}
Heuristically, the larger the degree of a rational map $f$, the more roots of unity $R$ must contain in order to admit a $\Mon_k(f)$-lift over $R$. This is why we consider lifts which preserve the action over an algebraic closure of $K$, as we are requiring lifts of rational maps of arbitrarily large degree.
\end{remark}

Given a rational map $f(z) \in k(z)$, $G$-lifts and geometric $\IMG$-lifts are related in the following way: $f$ admits a geometric $\IMG$-lift over $R$ if and only if there exists a sequence of rational maps $g_n(z) \in R(z)$ and a single $\varphi \in \Aut(\P^1_{\overline{K}})$ satisfying the following:
\begin{enumerate}
\item $g_n$ is a $\Mon_{\overline{k}}(f^n)$-lift of $f^n$ over $R$;
\item $g_n^\varphi = g_1^n$.
\end{enumerate}
In particular, the existence of a geometric $\IMG$-lift of $f$ implies the existence of a $\Mon_{\overline{k}}(f^n)$-lift of each iterate $f^n$. A priori, the converse is not true. For a geometric $\IMG$-lift to exist, there must exist a sequence of $\Mon_{\overline{k}}(f^n)$-lifts which is compatible, in the sense that it is linearly conjugate via a uniform change of variables to one given by iteration.

Optimistically, one might expect that $\IMG$-lifts exist whenever each iterate $f^n$ admits a $\Mon_{\overline{k}}(f^n)$-lift. By the Oort conjecture, every $f^n$ admits a $\Mon_{\overline{k}}(f^n)$-lift whenever all the inertia groups of each iterate $f^n$ are cyclic. We propose the following conjecture which reformulates this observation in terms of the profinite inertia groups defined in Section \ref{section:profiniteinertia}.

\begin{conjecture}[Dynamical Oort Conjecture]
Let $k$ be a perfect field of positive characteristic. Let $f(z) \in k(z)$ be a PCF rational map of degree $d \geq 2$ satisfying $f'(z) \neq 0$. Let $\mf{p}_1, \dots, \mf{p}_r$ denote the prime ideals of $\kbar(t)$ corresponding to the post-critical points of $f$. For each $1\leq i \leq r$, fix a prime ideal $\tilde{\mf{q}}_i$ of $\kbar \cdot K_\infty$ lying over $\mf{p}_i$. If all of the inertia groups $\mathcal{I}(\tilde{\mf{q}}_i/\mf{p}_i)$ are pro-cyclic then $f$ admits a lift to characteristic zero which preserves its geometric iterated monodromy group.
\end{conjecture}

A positive answer to Conjecture \ref{conj:dynamicaloort} for the family of tamely-ramified quadratic maps was given by Pink \cite{Pink13-Lifting}. In \cite{AdamsHyde25}, Adams and Hyde compute the geometric iterated monodromy group of every tamely ramified PCF unicritical polynomial, and find it depends on the structure of the post-critical orbit as well as some arithmetic data. Thus, in cases where reduction preserves the structure of the post-critical orbit, their paper provides a lift which preserves the geometric iterated monodromy group. Many computations of geometric iterated monodromy groups are stated this way, providing lifts for tamely ramified maps wherever reduction preserves the combinatorial data of the map. See for example \cite{Ejder26, Juul19}.

\section{Dynamics of Additive and Semi-Additive Polynomials}\label{section:additivepolys}

Suppose $f \in \overline{\F}_p[z]$ is additive. Then $f$ is of the form
\[f(z) = \sum_{i=0}^m a_i z^{p^i}\]
for some $a_i \in \overline{\F}_p$ \cite[Proposition 1.1.5]{Goss96}. Note that $f'(z) = a_0$. In particular, $f$ is separable if and only if $a_0 \neq 0$. One can use Fermat's Little Theorem, along with the standard form for additive polynomials, to show that $\lambda f(z) = f(\lambda z)$ for all $\lambda \in \F_p$.

\begin{lemma}
Suppose $f(z) \in \overline{\F}_p[z]$ is additive and separable. Then the following hold for each $n \geq 1$:
\begin{enumerate}
\item $f^n(z)$ is additive.
\item $f^n(z)$ is separable.
\end{enumerate}
\end{lemma}
\begin{proof}
Additivity follows from the more general fact that the collection of additive polynomials over $\overline{\F}_p$ forms a ring under addition and composition.
We prove separability by induction on $n$. Suppose $(f^{n-1})'(z) = a_0^{n-1}$. Using the chain rule, we have
\begin{align*}
(f^n)'(z) &= (f^{n-1})'(f(z))\cdot f'(z) \\
&= a_0^{n-1} \cdot a_0 \\
&= a_0^n.
\end{align*}
In particular, since $a_0 \neq 0$ we have $a_0^n \neq 0$. I.e, $f^n$ is separable. 
\end{proof}

We can leverage the additivity property to compute the geometric IMG of any additive, separable polynomial.

\begin{theorem}
Let $f(z) \in \overline{\F}_p[z]$ be an additive, separable polynomial of degree $p^m$. For each $n\geq 1$, we have an isomorphism
\[\Mon_{\overline{\F}_p}(f^n) \cong (\Z/p\Z)^{mn}.\]
The geometric iterated monodromy group then has the form
\[\IMG_{\overline{\F}_p}(f) \cong \varprojlim_{n\geq 1}(\Z/p\Z)^{mn}.\]
\end{theorem}

\begin{proof}
Let $n \geq 1$. Since $f^n$ is additive and commutes with scalar multiplication by $\F_p$, the roots of $f^n$ form an $\F_p$-vector subspace of $\overline{\F}_p$. Let $Z_n$ denote the set of roots of $f^n$ in $\overline{\F}_p$. Since $f^n$ is separable, we have that $|Z_n| = (p^m)^n$. In particular, $Z_n$ is isomorphic as an $\F_p$-vector space to $\F_p^{mn}$. 

We can fix $\theta_n \in \overline{\F_p(t)}$ such that
\[\overline{\F}_p(\theta_n) \cong \frac{\overline{\F}_p(t)[z]}{(f^n(z) - t)}.\]
For any $\alpha \in Z_n$, we can compute that
\[f^n(\theta_n + \alpha) = f^n(\theta_n) + f^n(\alpha) = t.\]
Since $f^n(z) - t$ is separable, the fiber $f^{-n}(t)$ has order $p^{mn}$. Thus, $f^{-n}(t)$ has the form
\[f^{-n}(t) = \{\theta_n + \alpha : \alpha \in Z_n\}.\]
In particular, $\overline{\F}_p(\theta_n) = K_n$ is the splitting field of $f^n(z) - t$. That is to say, each iterate $f^n$ is Galois as a cover of $\P^1_{\overline{\F}_p}$. Moreover, $Z_n$ maps isomorphically onto $\Mon_{\overline{\F}_p}(f^n) = \Gal(K_n/\overline{\F}_p(t))$ via the map
$\alpha \mapsto (\theta \mapsto \theta + \alpha)$.
\end{proof}

The geometric $\IMG$ of an additive polynomial is unique to positive characteristic, in the sense that it is not liftable to characteristic zero.

\begin{theorem}\label{thm:noadditivelift}
Let $f(z) \in \overline{\F}_p[z]$ be an additive, separable polynomial of degree $p^m$. Then $f$ does not admit a geometric $\IMG$-lift over any characteristic zero DVR.
\end{theorem}
\begin{proof}
Each iterate $f^n$ is Galois as a cover of $\P^1_{\overline{\F}_p}$, so that the monodromy group $\Mon_{\overline{\F}_P}(f^n)$ is a subgroup of $\Aut(\P^1_{\overline{\F}_p})$. Let $K$ be a local field of characteristic zero with algebraic closure $\overline{K}$. Any $\Mon_{\overline{\F}_p}(f^n)$-lift of $f^n$ to $K$ must also be Galois as a cover of $\P^1_{\overline{K}}$, requiring $\Mon_{\overline{\F}_p}(f^n)$ to be a subgroup of $\Aut(\P^1_{\overline{K}})$ as well. However, $(\Z/p\Z)^m$ is not a subgroup of $\Aut(\P^1_{\overline{K}})$ away from $m = 1$ or $(p,m) = (2,2)$ \cite{Beauville10}. So, in all cases, there exists an iterate $f^n$ which does not admit a $\Mon_{\overline{\F}_p}(f^n)$-lift to $K$ preventing a geometric $\IMG$-lift.
\end{proof}

As we have seen, additive polynomials over $\overline{\F}_p$ have $p$-elementary abelian monodromy groups. To realize $p$-semi-elementary monodromy groups, i.e. the semi-direct products of a $p$-elementary abelian group with a cyclic group of order coprime to $p$, we consider the collection of \textit{semi-additive} polynomials. Let $f(z) \in \overline{\F}_p[z]$ be an additive, separable polynomial. Given $n \geq 1$ with $p\nmid n$, we can form the \textit{semi-additive} polynomial $g(z) = z^n \circ f(z)$. Given certain divisibility conditions on the powers of $p$ arising in $f$, we can compute the monodromy groups of these semi-additive polynomials.

\begin{theorem}\label{thm:semiadditivemonodromy}
Let $n\geq 1$ such that $p\nmid n$. Let $e$ denote the order of $p$ in $(\Z/n\Z)^\times$. Fix $\ell \geq 1$ and let $f(z) \in \overline{\F}_p[z]$ be an additive, separable polynomial of the form
\[f(z) = \sum_{i=0}^\ell a_i z^{p^{ie}}\]
with $a_\ell \neq 0$. Let $g(z) = z^n \circ f(z)$. Then $\Mon_{\overline{\F}_p}(g) \cong (\Z/p\Z)^{\ell e} \rtimes \Z/n\Z$.
\end{theorem}

\begin{proof}
Fix $\theta \in \overline{\F_p(t)}$ such that
\[\overline{\F}_p(\theta) \cong \frac{\overline{\F}_p(t)[z]}{(g(z) - t)}.\]

Let $Z$ denote the set of roots of $f$ in $\overline{\F}_p$. Since $f$ is separable of degree $p^{\ell e}$, we know $Z$ is isomorphic as an $\F_p$-vector space to $\F_p^{\ell e}$. For any $\alpha \in Z$, we compute that
\[f(\theta + \alpha) = f(\theta) + f(\alpha) = f(\theta).\]
In particular, passing through the power map $z^n$ gives
\[g(\theta + \alpha) = g(\theta) = t.\]
It follows that $Z$ embeds in $\Mon_{\overline{\F}_p}(g)$. 

Let $\zeta_n \in \overline{\F}_p$ be a primitive $n$-th root of unity. Recall that $p^e \equiv 1 \mod{n}$. Fix $k \in \Z$ such that $nk = p^e - 1$. We can compute
\[\zeta_n^{p^e} = (\zeta_n^n)^k \cdot \zeta_n = \zeta_n.\]
By construction, for any $0 \leq i < n$, we have
$f(\zeta_n^i z) = \zeta_n^i f(z)$.
Thus, $g(\zeta_n^i\theta) = g(\theta) = t$ for any $0 \leq i < n$. It follows that $\Z/n\Z$ embeds in $\Mon_{\overline{\F}_p}(g)$. 

Since $g(z) - t$ is separable, the fiber $g^{-1}(t)$ has cardinality $p^{\ell e}n$. Hence, the fiber takes the form
\[g^{-1}(t) = \{\zeta_n^i (\theta + \alpha) : 0 \leq i < n,\ \alpha \in Z\}.\]
It follows that $\overline{\F}_p(\theta)$ is the splitting field of $g(z) - t$ and $g$ is Galois as a cover of $\P^1_{\overline{\F}_p}$. There is a well-defined surjective homomorphism $\Mon_{\overline{\F}_p}(g) \to \Z/n\Z$ taking $(\theta \mapsto \zeta_n^i(\theta + \alpha))$ to $i$. In fact, there is a short exact sequence
\[\begin{tikzcd}
1 \ar[r] & Z \ar[r] & \Mon_{\overline{\F}_p}(g) \ar[r] & \Z/n\Z \ar[r] & 1.
\end{tikzcd}\]
with section $\Z/n\Z \to \Mon_{\overline{\F}_p}(g)$ taking $i$ to $(\theta \mapsto \zeta_n^i\theta)$. Thus, $Z \rtimes \Z/n\Z$ maps isomorphically onto $\Mon_{\overline{\F}_p}(g)$ via the map $(\alpha, i) \mapsto (\theta \mapsto \zeta_n^i (\theta + \alpha))$.
\end{proof}

We will refer to a semi-additive polynomial of the form constructed in Theorem \ref{thm:semiadditivemonodromy} as being \textit{$e$-divisible}.

\begin{theorem}
Let $g(z) = z^n \circ f(z) \in \overline{\F}_p[z]$ be an $e$-divisible semi-additive polynomial of degree $p^{\ell e} n$. If $m = \ell e > 0$ and $(m,n) \notin \{(1,2), (2,3)\}$ then $g$ does not admit a geometric $\IMG$-lift over any characteristic zero DVR.
\end{theorem}
\begin{proof}
When $n = 1$, this is precisely the statement of Theorem \ref{thm:noadditivelift}. When $m > 0$, away from the dihedral case $(m,n) = (1,2)$ and the alternating case $(m,n) = (2,3)$, the group $(\Z/p\Z)^m \rtimes \Z/n\Z$ is not a subgroup of $\Aut(\P^1_{\overline{K}})$ for a field of characteristic zero $K$ \cite{Beauville10}. Failure to admit a geometric $\IMG$-lift then follows from the same argument used in the proof of Theorem \ref{thm:noadditivelift}.
\end{proof}

In \cite{Faber23}, Faber establishes that the only finite subgroups of $\Aut(\P^1_{\overline{\F}_p})$ which do not arise in characteristic zero are the $p$-semi-elementary-abelian groups, as well as $\mathrm{PGL}_2(\F_q)$ and $\mathrm{PSL}_2(\F_q)$ for $q = p^\ell$. The groups $\mathrm{PGL}_2(\F_q)$ and $\mathrm{PSL}_2(\F_q)$ can also be realized as monodromy groups of Galois rational maps using Dickson invariants \cite{Dickson11}. These realizations fail to admit geometric $\IMG$-lifts by the same arguments used above.

To compute the dimensions of the spaces of additive and semi-additive polynomials over $\overline{\F}_p$, we require an elementary lemma regarding normal forms for polynomials over algebraically closed fields.

\newpage\begin{lemma}\label{lem:normalforms}
Let $k = \overline{k}$ be an algebraically closed field of characteristic $p \geq 0$. Let $f(z) \in k[z]$ be a polynomial of degree $d \geq 2$. Then the set
\[S_f = \{\varphi \in \Aut(\P^1_k) : f^\varphi\text{ is a monic polynomial fixing }0\}\]
is finite.
\end{lemma}
\begin{proof}
Note that $\varphi \in S_f$ implies $\varphi(\infty)$ is a totally ramified fixed point of $f$. If $f$ has more than one totally ramified fixed point, then it has exactly two and $f$ is linearly conjugate to the power map $z^d$. In this case, the monic condition forces $S_{z^d} = \{z, 1/z\}$.

Suppose $f$ has exactly one totally ramified fixed point, the point at infinity. Any $\varphi \in S_f$ must fix $\infty$, so it has the form $\alpha z + \beta$ for some $\alpha, \beta \in k$. Fix $a_i \in k$ such that $f(z) = \sum_{i=0}^{d} a_iz^i$. Consider the following conditions:
\begin{itemize}
\item Monic: The leading coefficient of $f^\varphi$ is $\alpha^{d-1}a_d$. If $p \nmid (d-1)$, there are precisely $d-1$ choices of $\alpha \in k$ for which $\alpha^{d-1}a_d = 1$. If $p\mid (d-1)$, this number may decrease.
\item Fixing $0$: Recall that
\[f^\varphi(z) = \frac{1}{\alpha}f(\alpha z + \beta) - \frac{\beta}{\alpha}.\]
Setting $z= 0 $ yields
\[f^\varphi(0) = \frac{1}{\alpha}(f(\beta) - \beta).\]
Thus, $f^\varphi$ fixes $0$ if and only if $\beta \in k$ is a fixed point of $f$.
\end{itemize}
In a precise sense, a generic degree $d$ polynomial has $d$ distinct fixed points in $k$ \cite[Theorem 4.1 (iv)]{Silverman98}. Thus, when $p\nmid (d-1)$, $S_f$ generically contains $d(d-1)$ elements. This number decreases when $p\mid (d-1)$ or for specific choices of $f(z) \in k[z]$, but is always finite.
\end{proof}



Fix $n, \ell\geq 1$ with $p \nmid n$. Let $e$ denote the order of $p$ in $(\Z/n\Z)^\times$. We define a morphism $c: \A^{\ell - 1} \times (\A^1 \setminus \{0\}) \to \mathrm{Rat}_{p^{\ell e}n}$ given by
\[(a_{\ell - 1}, \cdots, a_1, a_0) \mapsto z^n \circ \left( z^{p^{\ell e}} + \sum_{i=0}^{\ell - 1} a_i z^{p^{ie}}\right).\]
Let $A_{n, \ell}$ denote the image of $c$ in $\mathrm{Rat}_{p^{\ell e}n}$. Note that $A_{n, \ell}(\overline{\F}_p)$ is precisely the set of separable, $e$-divisible, semi-additive polynomials over $\overline{\F}_p$. When $n=1$, we recover the collection of additive, separable polynomials of degree $p^\ell$. Let $\mathcal{A}_{n, \ell} = [A_{n, \ell}]$ denote the image of $A_{n, \ell}$ under the quotient morphism $\mathrm{Rat}_{p^{\ell e}n} \to \mathcal{M}_{p^{\ell e}n}$. Then $\mathcal{A}_{n, \ell}(\overline{\F}_p)$ is the collection of linear conjugacy classes of degree $p^{\ell e}n$ rational maps over $\overline{\F}_p$ which contain a separable, $e$-divisible, semi-additive polynomial.

Given $g(z) \in A_{n, \ell}(\overline{\F}_p)$, we can compute $g'(z) = na_0(f(z))^{n-1}$. Thus, $g$ has critical points at $\infty$ and at the roots of $f$ (when $n > 1$). While $\infty$ maps to itself with multiplicity $p^{\ell e}n$, the $p^{\ell e}$ roots of $f$ all map forward to $0$ with multiplicity $n$. The weighted-directed graph structure on the critical and post-critical points of $g$ is hence identical for every element $g(z) \in A_{n, \ell}(\overline{\F}_p)$. Despite this, the dimension of the collection of conjugacy classes $\mathcal{A}_{n, \ell}$ grows with $\ell$.

\begin{theorem}\label{thm:dimcount}
Fix $n, \ell \geq 1$ with $p\nmid n$. Then $\dim(\mathcal{A}_{n, \ell}) = \ell$. 
\end{theorem}
\begin{proof}
Let $[c]: \A^{\ell - 1} \times (\A \setminus \{0\})$ denote the composition of $c$ with the quotient morphism $[\cdot]$. Note that $[c]$ is surjective by construction. Since the elements of $A_{n, \ell}$ are all monic polynomials which fix zero, Lemma \ref{lem:normalforms} implies that $[c]$ is finite-to-one. Standard algebro-geometric arguments then imply that $[c]$ is dimension-preserving (see, for example, \cite[\href{https://stacks.math.columbia.edu/tag/02JX}{Lemma 02JX}]{stacks-project}).
\end{proof}

\section{An Exceptional Curve in $\mathcal{M}_p(\overline{\F}_p)$}\label{section:GMlifts}

We now focus on the case of additive, separable polynomials of degree $p$. The post-critical orbit of every additive, separable polynomial of degree $p^m$ consists of a single point mapping to itself with multiplicity $p^m$.
When $m = 1$, the converse is true as well, i.e. every rational map with precisely one critical point which is fixed and maps to itself with multiplicity $p$ must be conjugate to an additive polynomial. Let $\Gamma(\overline{\F}_p) \subseteq \mathcal{M}_p(\overline{\F}_p)$ denote the collection of linear conjugacy classes of degree $p$ rational maps with the described critical point and post-critical orbit.
\begin{proposition}
A complete set of distinct representatives for $\Gamma(\overline{\F}_p)$ is given by the following family:
\[\{f_c(z) = z^p - cz : c \in \overline{\F}_p^\times\}.\]
\end{proposition}
\begin{proof}
Given $f_c(z) = z^p - cz$, note that $f_c'(z) = -c$. Thus, when $c\neq 0$, the post-critical orbit $P_{f_c}$ has the desired form. Recall that the multiplier at a fixed point is invariant with respect to the action of $\Aut(\P^1_{\overline{\F}_p})$ \cite[Proposition 1.9]{Silverman07}. Since the derivative is constant, the multiplier of $f_c$ at any finite fixed point is $-c$. Since this value varies with $c$, distinct values of $c$ give rise to distinct linear conjugacy classes.

Let $g(z) \in \overline{\F}_p(z)$ be the representative of an element in $\Gamma(\overline{\F}_p)$. We can conjugate to assume the lone critical point of $g$ is the point at infinity, i.e. $g(z)$ is a polynomial of degree $p$. Moreover, since $g$ has no finite critical points, $g$ must be of the form
\[g(z) = a_pz^p + a_1z + a_0\]
for some $a_p, a_1, a_0 \in \overline{\F}_p$ with $a_p, a_1 \neq 0$. Fix elements $b,c \in \overline{\F}_p$ such that $b^{p-1} = a_p^{-1}$ and
$a_pc^p + (a_1-1)c + a_0 = 0$.
Let $\varphi(z) \in \Aut(\P^1_{\overline{\F}_p})$ be the map $\varphi(z) = bz + c$. One can check that
\[g^\varphi(z) = \varphi^{-1} \circ g \circ \varphi(z) = z^p + a_1z.\qedhere\]
\end{proof}

Fix a primitive $p$-th root of unity $\zeta_p \in \overline{\Q}_p$. Let $\lambda = \zeta_p - 1$. In \cite[Theorem 4.1]{Green-Matignon98}, Green and Matignon proved that the $(\Z/p\Z)$-cover of $\P^1_{\overline{\F}_p}$ given by $z^p - z = x$
lifts to the $(\Z/p\Z)$-cover of $\P^1_{\overline{\Q}_p}$ given by
\[\frac{(\lambda Z + 1)^p - 1}{\lambda^p} = X.\]

We use this result to inspire a lift of the first iterate of each polynomial $f_c \in \overline{\F}_p[z]$ to a polynomial $\tilde{f}_s \in \overline{\Q}_p[z]$, obtaining a curve in $\mathcal{M}_p(\overline{\Q}_p)$. As stated in Theorem \ref{thm:noadditivelift}, these newly obtained polynomials will not be geometric $\IMG$-lifts of the dynamical systems defined by the $f_c$. In fact, their failure to do so sheds light on how lifting to characteristic zero affects the dynamical invariants of a wildly ramified rational map.

\begin{proposition}\label{prop:GMlift}
Let $\overline{\Z}_p$ denote the integral closure of $\Z_p$ in $\overline{\Q}_p$. Fix $c \in \overline{\F}_p^\times$. Given $s \in \overline{\Q}_p$, define the polynomial
\[\tilde{f}_s(z) = \frac{(\lambda z + s)^p - s^p}{\lambda^p} \in \overline{\Q}_p[z].\]
Suppose $s \in \overline{\Z}_p$ and $s^{p-1} \equiv c \mod{(\lambda)}$.
Then $\tilde{f}_s(z)$ reduces to $f_c(z) = z^p - cz$ modulo $(\lambda)$.
\end{proposition}

\begin{proof}
We use the binomial theorem to expand $\tilde{f}_s(z) = \sum_{i=0}^{p-1} \binom{p}{i}\frac{1}{\lambda^i}z^{p-i}s^i$.
Recall the factorization of ideals $p \Z_p[\zeta_p] = (\lambda)^{p-1}$ \cite[Lemma 1.4]{Washington82}. In particular, $\nu_\lambda(\binom{p}{i}) > 0$ whenever $0 < i < p - 1$. Thus, the coefficient of $z^i$ vanishes upon reduction modulo $(\lambda)$ whenever $0 < i < p - 1$. By construction, $\tilde{f}_s(z)$ is monic and so reduces to a monic polynomial modulo $(\lambda)$. The coefficient of the linear term $z$ in $\tilde{f}_s(z)$ is $s^{p-1}\frac{p}{\lambda^{p-1}}$. By assumption, $s^{p-1}$ reduces to $c$ modulo $(\lambda)$. Thus, it remains to show that $p/\lambda^{p-1} \equiv -1\mod{(\lambda)}$.

Recall the identity $\prod_{i=1}^{p-1} (1 - \zeta_p^i) = p$. We can pull a $\lambda$ out of each term on the left hand side to obtain the equality in $\Z_p[\zeta_p]$
\[\prod_{i=1}^{p-1} (\zeta_p^{i-1} + \cdots + \zeta_p + 1) = \frac{p}{\lambda^{p-1}}.\]
Note that $\zeta_p \equiv 1\mod{(\lambda)}$ and so reducing modulo $(\lambda)$ gives
\[\frac{p}{\lambda^{p-1}} \equiv \prod_{i=1}^{p-1} i \equiv -1\mod{(\lambda)}.\qedhere\]
\end{proof}

We can compute the derivative
\[\tilde{f}_s'(z) = \frac{p}{\lambda^{p-1}}(\lambda z + s)^{p-1}.\]
Upon lifting, the lone totally ramified critical point of $f_c$ splits into two totally ramified critical points at $\infty$ and $-s/\lambda$. Thus, the post-critical orbit $P_{\tilde{f}_s}$ has the form
\[\begin{tikzcd}
\infty \arrow[loop, distance=2em, in=35, out=325, "p"'] &  & \\
-\frac{s}{\lambda} \ar[r, "p"] & -\frac{s^p}{\lambda^p} \ar[r] & \left(\frac{s}{\lambda} - \frac{s^p}{\lambda^p}\right)^p - \frac{s^p}{\lambda^p} \ar[r] & \cdots
\end{tikzcd}\]

\begin{proposition}\label{prop:onediml}
Let $s, s' \in \overline{\Q}_p^\times$. Fix a primitive $(p-1)$-st root of unity $\zeta_{p-1} \in \overline{\Q}_p$. Define $C: \overline{\Q}_p \to \overline{\Q}_p$ given by
\[C(s) = \frac{s}{\lambda} - \frac{s^p}{\lambda^p}.\]
Then $\tilde{f}_s$ is linearly conjugate to $\tilde{f}_{s'}$ if and only if $C(s') = \zeta_{p-1}^i C(s)$ for some $0 \leq i < p-1$. 
\end{proposition}
\begin{proof}
Let $\varphi_s(z) = z - \frac{s}{\lambda}$. We can compute
\begin{align*}
\tilde{f}_s^\varphi(z) &= \frac{(\lambda(z-\frac{s}{\lambda}) + s)^p - s^p}{\lambda^p} + \frac{s}{\lambda} \\
&= z^p + C(s).
\end{align*}
Let $g_s(z) = \tilde{f}_s^{\varphi_s}(z) = z^p + C(s)$. Away from $C(s) = C(s') = 0$, we have $g_s = g_{s'}^\psi$ if and only if $\psi$ fixes both $0$ and $\infty$. In particular, we need only consider $\psi(z) = \alpha z$ for some $\alpha \in \overline{\Q}_p^\times$. Since $g_s$ and $g_{s'}$ are monic, the proof of Lemma \ref{lem:normalforms} shows that $\alpha$ must be a $(p-1)$-st root of unity. Moreover, if $\alpha^{p-1} = 1$ and $\psi(z) = \alpha z$ then $g_s^\psi = z^p + \alpha C(s)$.
\end{proof}

For uncountably many values of $s \in \overline{\Z}_p$, the dynamical behavior of $\tilde{f}_s$ is very different from that of the reduction $f_c$. Recall the monodromy action at each level is transitive, and thus $\#\Mon_{\overline{k}}(g^n) \geq d^n$ for any rational $g(z) \in k(z)$ of degree $d$ with nonzero derivative. The maps $f_c$ achieve this lower bound for each $n$, partially as a result of their being post-critically finite. Most of the maps $\tilde{f}_s$ are post-critically infinite, and the order of their monodromy groups will far exceed this bound for each iterate.

\begin{proposition}\label{prop:liftisbad}
For all $s \in \overline{\Q}_p^\times$ outside of a countably infinite set, $\tilde{f}_s$ is post-critically infinite.
\end{proposition}
\begin{proof}
By Proposition \ref{prop:onediml}, it suffices to show that $g_s(z)$ is post-critically infinite for all $s \in \overline{\Q}_p^\times$ outside of a countably infinite set. Moreover, since $C(s)$ has finite fibers, we can consider the values $C \in \overline{\Q}_p$ for which the polynomial $F_C(z) = z^p + C$ is post-critically infinite.

If $F_C$ is PCF, then the critical point $0$ has finite forward orbit. In particular, we can fix positive integers $n > m \geq 0$ such that $F_C^m(0)$ has primitive period $n$. I.e., $C \in \overline{\Q}_p$ satisfies the relation $F_C^m(0) = F_C^{m+n}(0)$. Note that for any $n > m \geq 0$, the polynomials $F_C^m(0)$ and $F_C^{m+n}(0)$ have differing degrees in $C$, hence are nonzero. It follows that the collective locus of vanishing of the polynomials
\[\{F_C^m(0) = F_C^{m+n}(0) : n > m \geq 0\}\]
is a countable set which contains the collection of $C \in \overline{\Q}_p$ for which $F_C$ is PCF.
\end{proof}

Since $\tilde{f}_s$ is a tamely ramified unicritical polynomial, its geometric iterated monodromy group has been computed by Adams and Hyde \cite{AdamsHyde25}. When $\tilde{f}_s$ is post-critically infinite, we have $\IMG_{\overline{\Q}_p}(\tilde{f}_s) \cong \varprojlim \left[\Z/p\Z\right]^n$,
where $[\Z/p\Z]^n$ denotes the $n$-fold iterated wreath product \cite[Proposition 3.14]{AdamsHyde25}. In this case, one can compute that $\log_p(\#\Mon_{\overline{\F}_p}(f_c^n)) = n$ while
\[\log_p(\#\Mon_{\overline{\Q}_p}(\tilde{f}_s^n)) = p^{n-1} + \cdots + p + 1.\]

\printbibliography

\end{document}